\documentclass[11pt]{article}
\usepackage[english]{babel}
\usepackage[latin1]{inputenc}
\usepackage{latexsym, amsfonts}
\usepackage{amssymb, amsmath}
\usepackage{amsthm}
\usepackage[table]{xcolor}
\usepackage{eucal}

\def\dom{\mathop{\mathrm{Dom}}\nolimits}
\def\im{\mathop{\mathrm{Im}}\nolimits}
\def\ker{\mathop{\mathrm{Ker}}\nolimits} 
\def\rank{\mathop{\mathrm{rank}}\nolimits}

\def\N{\mathbb N}
\newtheorem{theorem}{Theorem}
\newtheorem{lemma}[theorem]{Lemma}

\newtheorem{proposition}[theorem]{Proposition}

\date{}

\begin{document}

\title{On monoids of monotone partial transformations of a finite chain whose domains and ranges are intervals}

\author{Hayrullah Ay\i k, V\'\i tor H. Fernandes~and Emrah Korkmaz}

\maketitle

\begin{abstract} 
In this note, we consider the monoid $\mathcal{PIM}_{n}$ of all partial monotone transformations on a chain with $n$ elements whose domains and ranges are intervals and its submonoid $\mathcal{IM}_{n}$ constituted by the full transformations. 
For both of these monoids, our aim is to determine their cardinalities and ranks and define them by means of presentations. 
We also calculate the number of nilpotent elements of $\mathcal{PIM}_{n}$. 
\end{abstract}


\medskip 


\section*{Introduction} 

We begin by defining the monoids that are the focus of this paper. For a positive integer $n$,  let $\Omega_{n}$ be a chain with $n$ elements, say $\Omega_{n}=\{1<2<\cdots<n\}$. It is usual to denote by $\mathcal{PT}_{n}$ the monoid of all partial transformations on $\Omega_{n}$ (under composition) and by $\mathcal{T}_{n}$ the submonoid of $\mathcal{PT}_{n}$ of all full transformations on $\Omega_{n}$. Notice that, we compose the transformations from left to right.

A transformation $\alpha \in \mathcal{PT}_{n}$ is called \emph{order-preserving} [\emph{order-reversing}] if, for all $x,y \in \dom(\alpha)$, $x < y$ implies $x\alpha\leqslant y\alpha$ [$x\alpha\geqslant y\alpha$], and is called \emph{monotone} if $\alpha$ is order-preserving or order-reversing. Clearly, the product of two order-preserving transformations or two order-reversing transformations is an order-preserving transformation and the product of an order-preserving transformation by an order-reversing transformation, or vice-versa, is an order-reversing transformation. Moreover, the product of two monotone transformations is monotone. 
We denote by $\mathcal{PO}_{n}$ [$\mathcal{O}_{n}$] the submonoid of $\mathcal{PT}_{n}$ [$\mathcal{T}_{n}$] of all order-preserving partial [full] transformations and by $\mathcal{PM}_{n}$ [$\mathcal{M}_{n}$] the submonoid of $\mathcal{PT}_{n}$ [$\mathcal{T}_{n}$] of all monotone partial [full] transformations. 

A subset $I$ of $\Omega_{n}$ is called an \emph{interval} of $\Omega_{n}$ if for all $x,y,z\in \Omega_{n}$,\, $x\leqslant y\leqslant z$ and $x,z\in I$ imply $y\in I$. We denote by $\mathcal{PIO}_{n}$ the submonoid of $\mathcal{PO}_{n}$ of all order-preserving partial transformations whose domain and image sets are both intervals of $\Omega_{n}$ and  by $\mathcal{IO}_{n}$ the submonoid $\mathcal{O}_{n}$ of all order-preserving transformations whose image sets are intervals of $\Omega_{n}$. 
Next, let us denote by $\mathcal{PIM}_{n}$ the subset of $\mathcal{PM}_{n}$ consisting of all monotone partial transformations whose domain and image sets are both intervals of $\Omega_{n}$. It is easy to show that $\mathcal{PIM}_{n}$ is a submonoid of $\mathcal{PM}_{n}$. Moreover, let 
$\mathcal{IM}_{n}$ be the submonoid of $\mathcal{PIM}_{n}$ of all monotone full transformations whose image sets are intervals of $\Omega_{n}$, 
i.e. $\mathcal{IM}_{n}=\mathcal{PM}_{n}\cap \mathcal{T}_{n}$.

\smallskip

Recall that the \textit{rank} of a finite monoid $M$, $\rank(M)$, is the minimum size of a generating set for $M$, i.e. 
$\rank(M)=\min\{\lvert X\rvert\mid X\subseteq M,\, \langle X\rangle=M\}$.
For a comprehensive background on semigroups and monoids, readers are referred to the textbook by Howie \cite{Howie:1995}.

\smallskip

It is well-known that $\mathcal{T}_{n}$ and $\mathcal{PT}_{n}$ have ranks $3$ and $4$, respectively. In \cite{Gomes&Howie:1992}, Gomes and
Howie showed that the ranks of the monoids $\mathcal{O}_{n}$ and $\mathcal{PO}_{n}$ are $n$ and
$2n-1$, respectively. In \cite{1Fernandes&Gomes&Jesus:2005}, Fernandes et al. showed that the ranks of $\mathcal{PM}_{n}$ and $\mathcal{M}_{n}$ are  $n+1$ and $\lceil \frac{n}{2}\rceil+1$, respectively, where $\lceil x \rceil$ denotes the least integer greater than or equal to a real number $x$. 

In $1962$,  A\u\i zen\v stat   \cite{Aizenstat:1962} gave a monoid presentation for $\mathcal{O}_{n}$ with $2n-2$ generators and $n^{2}$ relations. 
In the same year, Popova \cite{Popova:1962}  gave a monoid presentation for $\mathcal{PO}_{n}$ with $3n-2$ generators and $\frac{1}{2}(7n^{2}-n-4)$ relations 
(see also \cite{Fernandes:2002survey}). 
In $2005$, Fernandes et. al. \cite{Fernandes&Gomes&Jesus:2005} provided a monoid presentation for $\mathcal{PM}_{n}$ 
with $\lceil \frac{n}{2}\rceil+n$ generators and 
$\frac{1}{4}(7n^{2}+2n+\frac{3}{2}(1-(-1)^{n}))$ relations 
and a monoid presentation for $\mathcal{M}_{n}$ with $n$ generators and $\frac{1}{2}(n^{2}+n+2)$ relations.  
A significant amount of research (e.g. \cite{Bugay:2020,Fernandes&Santos:2019,Garba:1994,Gyudzhenov&Dimitrova:2006,Koppitz&Worawiset:2022,Li&Fernandes:2024,Yang:2000,Zhao&Fernandes:2015}) has extensively examined topics closely aligned with the focus of the present work.  

In \cite{Fernandes&Paulista:2023}, Fernandes and Paulista considered the monoid $\mathcal{IO}_n$.  
They showed that $\mathcal{IO}_{n}$ coincides with the monoid of all \emph{weak endomorphisms} of a a directed path with $n$ vertices. They also showed that the rank of $\mathcal{IO}_{n}$ is $n-1$. Building upon this work, Fernandes \cite{Fernandes:2024} gave a monoid presentation for  $\mathcal{IO}_{n}$ with $n-1$ generators and $\frac{1}{2}(3n^{2}-7n+4)$ relations. After then, in \cite{Ayik&al:2025}, Ay\i k et al. showed that the rank of $\mathcal{PIO}_{n}$ is $n+1$ and gave a monoid presentation  for $\mathcal{PIO}_{n}$  with $4n-4$ generators and $5n^{2}+3n-10$ relations.  

The present paper continues in the spirit of the research program outlined above. We give presentations for the monoid $\mathcal{PIM}_{n}$, 
in terms of  $4n-3$ generators and $5n^{2}+5n-10$ relations, and for the monoid $\mathcal{IM}_{n}$, in terms of $2n-1$ generators and $\frac{1}{2}(3n^{2}+n)$ relations. Moreover, we determine the cardinalities and ranks of these two monoids. In addition, we also characterize and count the nilpotent elements of $\mathcal{PIM}_{n}$.

\smallskip

We would like to acknowledge the use of computational tools, namely GAP \cite{GAP4}.

\section{Combinatorial and algebraic properties}

In this section, we collect some combinatorial and algebraic properties of the monoids $\mathcal{PIM}_{n}$ and $\mathcal{IM}_{n}$. 
It is worth recalling that the cardinalities of the monoids  $\mathcal{PIO}_{n}$ and $\mathcal{IO}_{n}$ have been calculated as  $(n+3)2^{n} -n^{2}-3n-2$ 
in \cite[Theorem 3]{Ayik&al:2025} and $(n+1)2^{n-2}$ in \cite[Theorem 2.6]{Fernandes&Paulista:2023} (see also \cite{Fernandes:2024}), respectively.

Let $\mathcal{PIM}_{n}^{r} = \{\alpha\in \mathcal{PIM}_{n}\mid \mbox{$\alpha$ is order-reversing}\}$ and 
$\mathcal{IM}_{n}^{r} = \mathcal{PIM}_{n}^{r} \cap \mathcal{T}_{n}$. Consider the permutation of order two 
$$
h= \left(\begin{array}{cccc}
	1 & 2 & \cdots   & n \\
	n & n-1 & \cdots & 1
\end{array}\right). 
$$ 
Observe that, for each $\alpha \in \mathcal{PIM}_{n}^{r}$ [$\alpha\in \mathcal{IM}_{n}^{r}$], it is clear that $\alpha h\in \mathcal{PIO}_{n}$ [$\alpha h\in \mathcal{IO}_{n}$]  and $\dom(\alpha) =\dom(\alpha h)$. Our first result follows.

\begin{theorem}\label{2} 
	For $n\geqslant 1$, $\lvert \mathcal{PIM}_{n}\rvert=(n+3)2^{n+1} -\frac{n^{3} +5n^{2}+ 12n+10}{2}$ and $\lvert \mathcal{IM}_{n} \rvert= (n+1)2^{n-1} -n$.
\end{theorem}

\begin{proof}
Let us define a mapping $\phi: \mathcal{PIM}_{n}^{r} \rightarrow \mathcal{PIO}_{n}$ by $\alpha \phi= \alpha h$, for all $\alpha\in \mathcal{PIM}_{n}^{r}$. 
Clearly, $\phi$ is a bijection, whence $\lvert \mathcal{PIM}_{n}^{r} \rvert= \lvert \mathcal{PIO}_{n} \rvert$.  
Moreover, since $\mathcal{PIM}_{n}^{r} \cap \mathcal{PIO}_{n}$ consists only of constant mappings and there exist $\frac{n(n+1)}{2}$ non-empty intervals of $\Omega_{n}$, by including the empty transformation $0_{n}$, 
we get $\lvert \mathcal{PIM}_{n}^{r} \cap\mathcal{PIO}_{n} \rvert =\frac{n^{2}(n+1)}{2} +1$. 
Hence, from \cite[Theorem 3]{Ayik&al:2025}  it follows that 
$$\textstyle 
\lvert \mathcal{PIM}_{n} \rvert= 2\lvert \mathcal{PIO}_{n} \rvert -\frac{n^{2}(n+1)}{2}-1= (n+3)2^{n+1} -\frac{n^{3}+5n^{2}+12n+10}{2}.
$$ 
On the other hand, since $\phi$ is also a bijection from $\mathcal{IM}_{n}^{r}$ into $\mathcal{IO}_{n}$ and $\lvert \mathcal{IM}_{n}^{r} \cap \mathcal{IO}_{n} \rvert =n$,  from \cite[Theorem 2.6]{Fernandes&Paulista:2023} it follows that $\lvert \mathcal{IM}_{n}\rvert=2\lvert \mathcal{IO}_{n}\rvert -n=(n+1)2^{n-1} -n$.
\end{proof}

\smallskip 

Let $S$ be a semigroup with zero $0$. An element $s\in S$ is called \emph{nilpotent} if there exists a positive integer $k$ such that $s^{k}=0$. Let us denote by $N(S)$ the set of all nilpotent elements of $S$. 
Observe that, $N(S)$ might not be a subsemigroup of $S$. 

Obviously, $0_n\in\mathcal{PIM}_{n}$ and so $\mathcal{PIM}_{n}$ is a semigroup with zero $0_n$.  
In order to find the cardinality of $N(\mathcal{PIM}_{n})$,  we begin by providing a characterization of the nilpotent elements of $\mathcal{PIM}_{n}$ 
which belong to $\mathcal{PIM}_{n}^{r}$. 

\begin{lemma}\label{3}
Let $\alpha\in \mathcal{PIM}_{n}^{r}$. Then, $\alpha$ is nilpotent if and only if $\im(\alpha)\cap \dom(\alpha)=\emptyset$.
\end{lemma}

\begin{proof}
First, suppose that $\im(\alpha)\cap \dom(\alpha)\neq\emptyset$. Take 
$i\in \im(\alpha) \cap \dom(\alpha)$ and let $k\in\Omega_n$ be such that $k\alpha=i$. 
If $i\leqslant k$, then $[i,k]\subseteq\dom(\alpha)$ and $i=k\alpha\leqslant i\alpha$, whence $[i,k]\alpha=[i,i\alpha]$ and so $i\alpha\leqslant k$, 
which implies that $i\alpha\in\dom(\alpha)\cap\im(\alpha)$. Similarly, if $k\leqslant i$, we can show that $i\alpha\in\dom(\alpha)\cap\im(\alpha)$. 
It follows by iteration that $i\alpha^m\in\dom(\alpha)\cap\im(\alpha)$, for all $m\in\N$. 
Hence, $\im(\alpha)\cap \dom(\alpha)\neq\emptyset$ implies $\alpha^m\neq 0_n$, for all $m\in\N$, i.e. $\alpha$ is not nilpotent. 

Conversely, if $\im(\alpha)\cap \dom(\alpha)=\emptyset$, then $\dom(\alpha^2)=\emptyset$, i.e. $\alpha^2=0_n$, and so $\alpha$ is nilpotent, as required. 
\end{proof}

Now, remember we proved in \cite[Proposition 4]{Ayik&al:2025} that $\lvert N(\mathcal{PIO}_{n}) \rvert=2^{n+2} - n^2 - 3n - 3$. 
Therefore, we can now establish our next result. 

\begin{proposition}\label{4}
	$\lvert N(\mathcal{PIM}_n)\rvert=2^{n+2}-n^2-3n-3+2\sum\limits_{r=2}^{n-2} \sum\limits_{j=1}^{n-r+1} \sum\limits_{k=2}^{j-1}(j-k)\binom{r-1}{k-1}$.
\end{proposition}

\begin{proof}
Let $\alpha\in \mathcal{PIM}_{n}^{r}$ be a non-null nilpotent element of  $\mathcal{PIM}_{n}$. 
Then, $\dom(\alpha)=[j,j+r-1]$, for some $1\leqslant r\leqslant n-1$ and $1\leqslant j\leqslant n-r+1$, 
and it is not difficult to see that there are 
$$\textstyle 
\sum\limits_{k=1}^{j-1}(j-k)\binom{r-1}{k-1}+ \sum\limits_{k=1}^{n-j-r+1} (n-j-r-k+2) \binom{r-1}{k-1}
$$ 
possibilities for transformations $\alpha$ with such domain. 
Since $|\im(\alpha)|=1$ implies $\alpha \in N(\mathcal{PIO}_{n})$, we conclude that 
$$
\begin{array}{rcl} 
	N(\mathcal{PIM}_{n}^{r}) \setminus N(\mathcal{PIO}_{n}) &= & 
	\sum\limits_{r=2}^{n-2} \sum\limits_{j=1}^{n-r+1} \left[ \sum\limits_{k=2}^{j-1} (j-k) \binom{r-1}{k-1} + 
	\sum\limits_{k=2}^{n-j-r+1} (n-j-r-k+2) \binom{r-1}{k-1} \right] \\[12pt]
	&= & 2 \sum\limits_{r=2}^{n-2} \sum\limits_{j=1}^{n-r+1} \sum\limits_{k=2}^{j-1} (j-k) \binom{r-1}{k-1}. 
\end{array}
$$
Therefore, we obtain
$$ 
\begin{array}{rcl} 
	\lvert N(\mathcal{PIM}_n) \rvert &= & \lvert N(\mathcal{PIO}_{n}) \rvert + \lvert N(\mathcal{PIM}_{n}^{r}) \setminus N(\mathcal{PIO}_{n}) \rvert \\ [4pt]
	&= & 2^{n+2} - n^2 - 3n - 3 + 2\sum\limits_{r=2}^{n-2} \sum\limits_{j=1}^{n-r+1} \sum\limits_{k=2}^{j-1} (j-k) \binom{r-1}{k-1},
	\end{array}
$$
as claimed.	
\end{proof}

\begin{center}
	\renewcommand{\arraystretch}{1.5}
	\begin{tabular}{|c|*{12}{c|}}
		\hline
		$n$ & 1 & 2 & 3 & 4 & 5 & 6 & 7 & 8 & 9 & 10 & 11 & 12 \\
		\hline
		$\lvert \mathcal{IM}_n \rvert$ & 1 & 4 & 13 & 36 & 91 & 218 & 505 & 1144 & 2551 & 5622 & 12277 & 26612   \\
		\hline
		$\lvert \mathcal{PIM}_n\rvert$ & 2 & 9 & 37 & 123 & 352 & 913 & 2219 & 5163 & 11662 & 25809 & 56305 & 121579 \\
		\hline
		$\lvert N(\mathcal{PIM}_n) \rvert$ & 1 & 3 & 11 & 35 & 95 & 231 & 521 & 1117 & 2315 & 4693 & 9395 & 18523 \\
		\hline
	\end{tabular}
\end{center}
\smallskip	

Notice that, for $n\geqslant 2$, it is easy to verify that $N(\mathcal{PIM}_{n})$  is not a semigroup. 

\smallskip 

Next, we will determine the ranks of the monoids $\mathcal{IM}_{n}$ and  $\mathcal{PIM}_{n}$. 

First, we consider the monoid $\mathcal{IM}_{n}$. For $n\geqslant 3$ and $1\leqslant i\leqslant n-1$, let
$$
	a_i = \left(\begin{array}{cccccc}
		1 & \cdots & i & i+1 & \cdots   & n \\
		1 & \cdots & i & i   & \cdots   & n-1
	\end{array}\right) \quad \mbox{and} \quad 
	b_i = \left(
	\begin{array}{cccccccc}
		1 & \cdots & i   & i+1 & \cdots   & n \\
		2 & \cdots & i+1 & i+1 & \cdots   & n
	\end{array}\right).
$$
It is shown in \cite[Proposition 3.3] {Fernandes&Paulista:2023} that $\{a_1, \ldots, a_{n-2}, b_{n-1}\}$ is a generating set of $\mathcal{IO}_{n}$ with minimum size. 
Notice that, $\mathcal{IM}_{1}=\mathcal{T}_{1}$ and it is
routine matter to check that
$\mathcal{IM}_{2} =\langle  \left( \begin{smallmatrix}
1 & 2 \\
1 & 1
\end{smallmatrix}\right),  \left( \begin{smallmatrix}
1 & 2 \\
2 & 1
\end{smallmatrix}\right) \rangle$. Hence, clearly, $\mathcal{IM}_{1}$ has rank $0$ and $\mathcal{IM}_{2}$ has rank $2$.
Let us consider $n\geqslant 3$ and define
$$
	c_{i} = \left(\begin{array}{cccccccc}
		1  & \cdots & n-i-1 & n-i & n-i+1 & n-i+2 & \cdots   & n \\
		n-1 & \cdots &  i+1  &  i  &   i   &  i-1  & \cdots   & 1
	\end{array}\right),
$$
for $1\leqslant i\leqslant \lfloor \frac{n}{2} \rfloor$, 
where $\lfloor x \rfloor$ denotes the greatest integer less than or equal to a real number $x$. 

\begin{lemma}\label{5}
Let $n\geqslant 3$. Then, $\{c_{1}, \ldots, c_{\lfloor \frac{n}{2} \rfloor},  h\}$ is a generating set of $\mathcal{IM}_{n}$.
\end{lemma}

\begin{proof} 
Since $\alpha=(\alpha h)h$ and $\alpha h\in \mathcal{IO}_{n}$, for all $\alpha\in  \mathcal{IM}_{n}^{r}$, 
we conclude that $\mathcal{IM}_{n}$ is generated by  $\mathcal{IO}_{n} \cup \{h\}$. 
On the other hand, it is a routine matter to check that 
$$
b_{n-1}=c_{1}h, \quad a_{i}=hc_{i} \quad \mbox{and}\quad a_{n-i}=c_{i}c_{1}, \quad \mbox{for $1\leqslant i\leqslant \lfloor \frac{n}{2} \rfloor$}.
$$ 
Hence, $\mathcal{IM}_{n}=\langle a_1,\ldots, a_{n-2},b_{n-1},h\rangle \subseteq \langle c_{1}, \ldots, c_{\lfloor \frac{n}{2} \rfloor},  h \rangle$, and so 
 $\{c_{1}, \ldots, c_{\lfloor \frac{n}{2} \rfloor},  h \}$ generates $\mathcal{IM}_{n}$, as required.
\end{proof}

In order to prove that $\{c_{1}, \ldots, c_{\lfloor \frac{n}{2} \rfloor},  h\}$ is a  generating set  of $\mathcal{IM}_{n}$ with minimum size, 
let $D_{r}=\{ \alpha \in \mathcal{IM}_{n} \mid \lvert \im(\alpha)\rvert =r \}$, for $r\in\{n-1,n\}$, and
let $\pi_i$ be the equivalence on $\Omega_{n}$ defined by the partition 
$\{ \{1\},\ldots ,\{i-1\}, \{i,i+1\}, \{i+2\},\ldots \{n\}\}$, for $1\leqslant i\leqslant n-1$. 
It is clear that $D_{n}=\{1_{n},h\}=\langle h\rangle$, where $1_n$ denotes the identity transformation on $\Omega_n$. 
On the other hand, it is easy to check that 
$
D_{n-1}=\{a_i,ha_i,a_ih,ha_ih\mid 1\leqslant i\leqslant n-1\}
$
and, for $1\leqslant i\leqslant n-1$, we have $\ker(a_i)=\ker(ha_ih)=\pi_i$ and $\ker(ha_i)=\ker(a_ih)=\pi_{n-i}$. 

\begin{theorem}\label{6}
Let $n\geqslant 3$. Then,  $\rank(\mathcal{IM}_{n})=\lfloor \frac{n}{2}\rfloor +1$.
\end{theorem}

\begin{proof} 
Let $U$ be any generating set of $\mathcal{IM}_{n}$. It is clear that $h\in U$ and $U\cap D_{n-1} \neq \emptyset$, since $\langle h\rangle =D_{n}$. 
Let $1\leqslant i\leqslant \lfloor \frac{n}{2} \rfloor$. 
Then, there exist $\alpha_{1}, \ldots ,\alpha_{t} \in U$ such that $a_i= \alpha_{1} \cdots \alpha_{t}$. 
If $t=1$, then $a_i\in U$ and $\ker(a_i)=\pi_i$. Suppose that $t\geqslant2$. Since $h^2=1_n$, we can also suppose that $\alpha_1\neq h$ or $\alpha_2\neq h$. 
If $\alpha_1\neq h$, then $\alpha_1\in D_{n-1}$ and $\ker(\alpha_1)\subseteq\ker(a_i)=\pi_i$, whence $\ker(\alpha_1)=\pi_i$. 
On the other hand, if $\alpha_1=h$, then $\alpha_2\neq h$ and $ha_i=\alpha_2\cdots\alpha_t$, 
whence $\alpha_2\in D_{n-1}$ and $\ker(\alpha_2)\subseteq\ker(ha_i)=\pi_{n-i}$, and so $\ker(\alpha_2)=\pi_{n-i}$. 
Therefore, clearly, $U$ contains at least $\lfloor \frac{n}{2} \rfloor$
distinct elements of $D_{n-1}$, whence $\lvert  U\rvert \geqslant \lfloor \frac{n}{2} \rfloor+1$, and so, by Lemma \ref{5}, 
we get $\rank(\mathcal{IM}_{n})=\lfloor \frac{n}{2}\rfloor +1$, as required. 
\end{proof}

Now, we consider the monoid $\mathcal{PIM}_{n}$. 
Notice that, $\mathcal{PIM}_{1}=\mathcal{PT}_{1}$ and so $\mathcal{PIM}_{1}$ has rank $1$. 
On the other hand, it is easy to check that
$ \left\{ \left( \begin{smallmatrix}
1  \\
1 
\end{smallmatrix}\right),  \left( \begin{smallmatrix}
1 & 2 \\
1 & 1
\end{smallmatrix}\right),   \left( \begin{smallmatrix}
1 & 2 \\
2 & 1
\end{smallmatrix}\right)\right\}$ 
is a generating set of $\mathcal{PIM}_{2}$ with minimum size, whence $\mathcal{PIM}_{2}$ has rank $3$. 
For $n\geqslant 3$ and $1\leqslant i\leqslant n-1$, define the following elements of $\mathcal{PIM}_{n}$: 
\begin{equation*} 
	e_{i}=\begin{pmatrix} 
		1 & \cdots & i  \\
		1 & \cdots & i 
	\end{pmatrix} \quad\text{ and }\quad
	f_{i+1}=\begin{pmatrix} 
		i+1 & \cdots   & n \\
		i+1 & \cdots   & n
	\end{pmatrix}.
\end{equation*}
Recall that, in \cite[Theorem 11]{Ayik&al:2025}, we have shown that $\{a_1, \ldots, a_{n-2}, b_{n-1}, e_{n-1}, f_{2}\}$ is a minimal generating set of $\mathcal{PIO}_{n}$. 

\begin{lemma}\label{5p}
Let $n\geqslant 3$. Then, $\{c_{1}, \ldots, c_{\lfloor \frac{n}{2} \rfloor},  e_{n-1},h\}$ is a generating set of $\mathcal{PIM}_{n}$.
\end{lemma}

\begin{proof} 
Let $\alpha\in  \mathcal{PIM}_{n}^{r}$. Then, $\alpha=(\alpha h)h$ and $\alpha h\in \mathcal{PIO}_{n}$. 
This allows us to conclude that $\mathcal{PIM}_{n}$ is generated by $\mathcal{PIO}_{n} \cup \{h\}$ and so 
$\{a_1, \ldots, a_{n-2}, b_{n-1}, e_{n-1}, f_{2},h\}$ also generates  $\mathcal{PIM}_{n}$. 
Since $a_1, \ldots, a_{n-2}, b_{n-1}\in \langle c_{1}, \ldots, c_{\lfloor \frac{n}{2} \rfloor},  h \rangle$, by Lemma \ref{5}, 
and $f_2=he_{n-1}h$, it follows that  
$\mathcal{PIM}_{n}=\langle a_1,\ldots, a_{n-2},b_{n-1},e_{n-1},h\rangle \subseteq \langle c_{1}, \ldots, c_{\lfloor \frac{n}{2} \rfloor}, e_{n-1}, h \rangle$
and so 
 $\{c_{1}, \ldots, c_{\lfloor \frac{n}{2} \rfloor}, e_{n-1}, h \}$ generates $\mathcal{PIM}_{n}$, as required.
\end{proof}

\begin{theorem}\label{7}
Let $n\geqslant 3$. Then, $\rank(\mathcal{PIM}_{n})=\lfloor \frac{n}{2}\rfloor +2$.
\end{theorem}

\begin{proof}  
Let $V$ be any generating set of $\mathcal{PIM}_{n}$. 
Clearly, as for $\mathcal{IM}_{n}$, we must have $h\in V$.
Let $1\leqslant i\leqslant \lfloor \frac{n}{2} \rfloor$. 
Then, there exist $\alpha_{1}, \ldots ,\alpha_{t} \in V$ such that $a_i= \alpha_{1} \cdots \alpha_{t}$. 
Since $\dom(\alpha_1)=\dom(a_i)=\Omega_n$ and, if $t\geqslant2$, $\dom(\alpha_2)=\dom(ha_i)=\Omega_n$, 
the same reasoning as in the proof of Theorem \ref{6} allows us to deduce that $V$ also contains at least $\lfloor \frac{n}{2} \rfloor$
distinct elements of $D_{n-1}$. On the other hand, $V$ must contain at least a non-full transformation, 
whence $\lvert  V\rvert \geqslant \lfloor \frac{n}{2} \rfloor+2$, and so, by Lemma \ref{5p}, 
we have $\rank(\mathcal{PIM}_{n})=\lfloor \frac{n}{2}\rfloor +2$, as stated. 
\end{proof}

\section{Presentations}

We begin this section by recalling some notions on presentations. 
For a set $A$, let $A^{*}$ denote the free monoid on $A$ consisting of all finite words over $A$. 
The empty word is denoted by $1$. 
A \emph{monoid presentation} is an ordered pair $\langle A\mid R\rangle$, where $A$ is an alphabet and $R$ is a subset of $A^{*}\times A^{*}$. 
Each element $(u,v)$ of $R$ is called a \textit{(defining) relation}, and it is usually written by $u=v$. 
A monoid $M$ is said to be \textit{defined by a presentation} $\langle A\mid R\rangle$ if $M$ is isomorphic to $A^{*}/\sim_{R}$, where $\sim_R$ denotes the congruence on $A^*$ generated by $R$, i.e. $\sim_{R}$ is the smallest congruence on $A^{*}$ containing $R$. 
Let  $X$ be a generating set of a monoid $M$ and let $\phi: A\rightarrow M$ be an injective mapping such that $A\phi =X$. 
If $\varphi : A^{*}\rightarrow M$ is the (unique) homomorphism that extends $\phi$ to $A^{*}$, 
then we say that $X$ \textit{satisfy} a relation $u=v$ of $A^{*}$  if $u\varphi =v\varphi$. 
Usually, if there is no danger of ambiguity, we represent the set of generators of $M$ and the alphabet by the same symbol, as well as their elements, thus considering the mapping $\phi: A\rightarrow M$ such that $x\mapsto x$, for all $x\in A$. 
For more details, see \cite{Lallement:1979} or \cite{Ruskuc:1995}.

\smallskip

Next, we describe the process established in \cite{Fernandes&al:2004} to obtain a
presentation for a finite monoid $T$ given a presentation for a certain
submonoid of $T$. This method will be applied to obtain presentations for $\mathcal{IM}_n$ and $\mathcal{PIM}_n$. 

\smallskip

Let $T$ be a (finite) monoid with identity $1$, let $M$ be a submonoid of $T$ and let $y$ be an
element of $T$ such that $y^2=1$. Let us suppose that $T$ is
generated by $M$ and $y$. Let $X=\{x_1,\ldots,x_k\}$ ($k\in\N$) be
a generating set of $M$ and $\langle X\mid R\rangle$ a
presentation for $M$. 
Suppose there exists a \textit{set of canonical forms} $W$ for $\langle X\mid R\rangle$, i.e. 
a transversal for the congruence $\sim_R$ of $X^*$, 
two subsets $U$ and $V$ of $W$ and a word $u_0\in X^*$
such that $W=U\cup V$ and
$u_0$ is a factor of each word in $U$. 
Let $Y=X\cup\{y\}$ (notice that $Y$ generates $T$) and suppose there exist words
$v_0,v_1,\ldots,v_k\in X^*$ such that the following relations over
the alphabet $Y$ are satisfied by the generating set $Y$ of $T$:
\begin{description}
	\item $(\text{N}_1)$ $yx_i=v_iy$, for all $1\leqslant i\leqslant k$;
	\item $(\text{N}_2)$ $u_0y=v_0$.
\end{description}

Observe that the relation (over the alphabet $Y$)
\begin{description}
	\item $(\text{N}_0)$  $y^2=1$
\end{description}
is also satisfied (by the generating set $Y$ of $T$), by
hypothesis. Let
$$
\overline{R}=R\cup \text{N}_0\cup \text{N}_1\cup \text{N}_2
\quad\text{and}\quad 
\overline{W}=W\cup\{wy\mid w\in V\}\subseteq Y^*. 
$$
Then, we have the following result. 

\begin{theorem}[{\cite[Theorem 2.4]{Fernandes&al:2004}}] \label{overpresentation}
	Under the previous conditions, if $W$ contains the empty word, then $\overline{W}$ 
	is a set of canonical forms for the presentation $\langle Y\mid \overline{R}\rangle$. 
	Moreover, if $\lvert \overline{W}\rvert \leqslant \lvert T\rvert$, then the monoid $T$ is defined by the presentation 
	$\langle Y\mid\overline{R}\rangle$. 
\end{theorem}

\medskip

Throughout this section we consider $n\geqslant2$. 

\smallskip 

Let $A=\{a_1,\ldots,a_{n-1}\}$ and $B=\{b_1,\ldots,b_{n-1}\}$. Then, $A\cup B$ is a set of generators of $\mathcal{IO}_n$. 
Let us consider a presentation $\langle A\cup B\mid R\rangle$ for $\mathcal{IO}_n$ on these generators; for instance, 
the presentation established by Fernandes in \cite[Theorem 4.8]{Fernandes:2024}. 
Let $\varphi:(A\cup B)^*\rightarrow\mathcal{IO}_n$ be the surjective homomorphism extending the mapping $A\cup B\rightarrow \mathcal{IO}_n$, 
$a_i\mapsto a_i$ and $b_i\mapsto b_i$, for $1\leqslant i\leqslant n-1$. 
Let $W$ be a set of canonical forms for $\langle A\cup B\mid R\rangle$ and let 
$$
V=\{w\in W\mid |\im(w\varphi)|\geqslant2\}. 
$$
Take $u_0=a_1^{n-1}$. Then, $u_0\varphi=\left(\begin{smallmatrix} 1& \cdots & n\\ 1& \cdots & 1 \end{smallmatrix}\right)$ and so, clearly, 
$(u_0w)\varphi=w\varphi$, for all $w\in W\setminus V$. 
Let 
$$
U=\{u_0w\mid w\in W\setminus V\}.
$$ 
Hence, $U\cup V$ is also a set of canonical forms for $\langle A\cup B\mid R\rangle$ and $u_0$ is a factor of each word in $U$. 
Notice that $V$ must contain the empty word. 

Let $C=A\cup B\cup\{h\}$. Then, $C$ is a generating set of $\mathcal{IM}_n$. Let us consider the following relations over the alphabet $C$: 
\begin{description}
\item $(N_0)$  $h^2=1$; 
\item $(N_1)$ $ha_i=b_{n-i}h$, for $1\leqslant i\leqslant n-1$;
\item $(N_2)$ $a_1^{n-1}h=b_{n-1}^{n-1}$.
\end{description}
It is a routine matter to check that all relations from $N_0\cup N_1\cup N_2$ are satisfied by the generating set $C$ of $\mathcal{IM}_n$. 
Let
$$
\overline{R}=R\cup N_0\cup N_1\cup N_2
\quad\text{and}\quad 
\overline{W}=U\cup V\cup\{wh\mid w\in V\}\subseteq C^*. 
$$

Notice that, clearly, $|\overline{W}|=2|W|-|U|=2|\mathcal{IO}_n|-n=|\mathcal{IM}_n|$. 
Observe also that, for $1\leqslant i\leqslant n-1$, we have 
\begin{equation}\label{n1}
hb_i\sim_{N_0} h(b_ih)h \sim_{N_1} h(ha_{n-i})h \sim_{N_0} a_{n-i}h,
\end{equation} 
whence $hb_i \sim_{\overline{R}} a_{n-i}h$ and so, by Theorem \ref{overpresentation} and \cite[Theorem 4.8]{Fernandes:2024}, we are able to immediately conclude the following. 

\begin{theorem}\label{8}
The monoid $\mathcal{IM}_n$ is defined by the presentation $\langle C\mid \overline{R}\rangle$ on $2n-1$ generators and $\frac{1}{2}(3n^{2}+n)$ relations. 
\end{theorem} 

\medskip 

Next, in a similar way, we will get a presentation for $\mathcal{PIO}_n$. 

Let $E=\{e_1,\ldots,e_{n-1}\}$ and $F=\{f_2,\ldots,f_n\}$. 
Then, $A\cup B\cup E\cup F$ is a set of generators of $\mathcal{PIO}_n$. 
For instance, consider the presentation $\langle A\cup B\cup E\cup F\mid R'\rangle$ for $\mathcal{PIO}_n$ 
on these generators established by Ay\i k et al. in \cite[Theorem 31]{Ayik&al:2025}. 
Let $\psi:(A\cup B\cup E\cup F)^*\rightarrow\mathcal{PIO}_n$ be the surjective homomorphism 
that extends the mapping $A\cup B\cup E\cup F\rightarrow \mathcal{PIO}_n$, 
$a_i\mapsto a_i$, $b_i\mapsto b_i$, $e_i\mapsto e_i$ and $f_{i+1}\mapsto f_{i+1}$, for $1\leqslant i\leqslant n-1$. 
Let $W'$ be a set of canonical forms for $\langle A\cup B\cup E\cup F\mid R'\rangle$ and let 
$$
V'=\{w\in W'\mid |\im(w\psi)|\geqslant 2\}. 
$$
Notice that $V'$ must also contain the empty word.  
For $1\leqslant i\leqslant n$, $0\leqslant j\leqslant n-i$ and $1\leqslant k\leqslant n$, let $w_{i,j,k}\in W'$ be such that 
$$
w_{i,j,k}\psi = \begin{pmatrix} i& \cdots & i+j\\ k& \cdots & k \end{pmatrix}. 
$$
Let also $w_0\in W'$ be such that $w_0\psi=0_n$. 
Observe that 
$$
W'\setminus V' =\{w_{i,j,k}\mid 1\leqslant i\leqslant n, 0\leqslant j\leqslant n-i, 1\leqslant k\leqslant n\}\cup\{w_0\}. 
$$
Since $w_{i,j,k}\psi=(w_{i,j,1}u_0w_{1,0,k})\psi$, for $1\leqslant i\leqslant n$, $0\leqslant j\leqslant n-i$ and $1\leqslant k\leqslant n$, 
and $(u_0w_0)\psi=w_0\psi$, with 
$$
U'=\{w_{i,j,1}u_0w_{1,0,k}\mid 1\leqslant i\leqslant n, 0\leqslant j\leqslant n-i, 1\leqslant k\leqslant n\}\cup\{u_0w_0\}, 
$$
we get that $U'\cup V'$ is also a set of canonical forms for $\langle A\cup B\cup E\cup F\mid R'\rangle$ and $u_0$ is a factor of each word in $U'$. 

Now, let $D=A\cup B\cup E\cup F\cup\{h\}$. 
Then, $D$ is a generating set of $\mathcal{PIM}_n$. Consider the following relations over the alphabet $D$: 
\begin{description}
\item $(N'_1)$ $he_i=f_{n-i+1}h$, for $1\leqslant i\leqslant n-1$. 
\end{description}
We can routinely prove that all relations from $N'_1$ are satisfied by the generating set $D$ of $\mathcal{PIM}_n$ and so, 
by a previous observation,  
all relations from $N_0\cup N_1\cup N'_1\cup N_2$ are satisfied by the generating set $D$ of $\mathcal{PIM}_n$. 
Let
$$
\widehat{R}=R'\cup N_0\cup N_1\cup N'_1\cup N_2
\quad\text{and}\quad 
\widehat{W}=U'\cup V'\cup\{wh\mid w\in V'\}\subseteq D^*. 
$$
Since $|\widehat{W}|=2|W'|-|U'|=2|\mathcal{PIO}_n|-\frac{(1+n)n^2}{2}-1=|\mathcal{PIM}_n|$ and, 
in a similar way to (\ref{n1}), we get $hf_i  \sim_{\widehat{R}} e_{n-i+1}h$, for $2\leqslant i\leqslant n$, then 
we have all the conditions guaranteed to, in view of Theorem \ref{overpresentation} and \cite[Theorem 31]{Ayik&al:2025}, conclude that:

\begin{theorem}\label{9}
The monoid $\mathcal{PIM}_n$ is defined by the presentation $\langle D\mid \widehat{R}\rangle$ on $4n-3$ generators and $5n^{2}+5n-10$ relations. 
\end{theorem} 

\section*{Acknowledgements}  

The authors would like to thank the anonymous referee for his/her comments and observations. 

\smallskip 

\noindent The first author was supported by Scientific and Technological Research Council of Turkey (\textsc{tubitak}) under the Grant Number 123F463. 
The author thanks to \textsc{tubitak} for their supports. 

\noindent The second author was supported by national funds through the FCT-Funda\c c\~ao para a Ci\^encia e a Tecnologia, I.P., 
under the scope of the Center for Mathematics and Applications projects 
UIDB/00297/2020 (https://doi.org/10.54499/UIDB/00297/2020) and UIDP/00297/2020 \linebreak (https://doi.org/10.54499/UIDP/00297/2020).

\bigskip 

{\sf\small 
\noindent{\sc Hayrullah Ay\i k},
\c{C}ukurova University,
Department of Mathematics,
Sar\i \c{c}am, Adana,
Turkey;
e-mail: hayik@cu.edu.tr  

\medskip 

\noindent{\sc V\'\i tor H. Fernandes},
Center for Mathematics and Applications (NOVA Math)
and Department of Mathematics,
Faculdade de Ci\^encias e Tecnologia,
Universidade Nova de Lisboa,
Monte da Caparica,
2829-516 Caparica,
Portugal;
e-mail: vhf@fct.unl.pt.

\medskip 

\noindent{\sc Emrah Korkmaz}, 
\c{C}ukurova University,
Department of Mathematics,
Sar\i \c{c}am, Adana,
Turkey;
e-mail: ekorkmaz@cu.edu.tr 
}

\end{document}